\documentclass[11pt,oneside,english]{amsart}
\usepackage[latin9]{inputenc}
\usepackage{amstext}
\usepackage{amsthm}
\usepackage{amssymb}
\usepackage{graphicx}
\usepackage{setspace}
\makeatletter
\numberwithin{equation}{section}
\numberwithin{figure}{section}
\theoremstyle{plain}
\newtheorem{thm}{\protect\theoremname}
  \theoremstyle{plain}
  \newtheorem{cor}[thm]{\protect\corollaryname}
  \theoremstyle{plain}
  \newtheorem{lem}[thm]{\protect\lemmaname}
  \theoremstyle{plain}
  \newtheorem{prop}[thm]{\protect\propositionname}

\makeatother

\usepackage{babel}
  \providecommand{\corollaryname}{Corollary}
  \providecommand{\lemmaname}{Lemma}
  \providecommand{\propositionname}{Proposition}
\providecommand{\theoremname}{Theorem}

\begin{document}

\title[Counting one sided simple closed geodesics]{Counting one sided simple closed geodesics on fuchsian thrice punctured
projective planes}
\begin{abstract}
We prove that there is a true asymptotic formula for the number of
one sided simple closed curves of length $\leq L$ on any Fuchsian
real projective plane with three points removed. The exponent of growth
is independent of the hyperbolic structure, and it is noninteger, in contrast to counting results of Mirzakhani
for simple closed curves on orientable Fuchsian surfaces.
\end{abstract}

\author{michael magee}

\thanks{M. Magee was supported in part by NSF award DMS-1701357.}

\maketitle

\global\long\def\Z{\mathbb{\mathbf{Z}}}
\global\long\def\Spec{\mathrm{Spec}}
\global\long\def\G{\mathbb{G}}
\global\long\def\X{\mathbb{X}}
\global\long\def\GL{\mathrm{GL}}
\global\long\def\O{\mathcal{O}}
\global\long\def\SL{\mathrm{SL}}
\global\long\def\C{\mathbf{C}}
\global\long\def\tr{\mathrm{tr}}
\global\long\def\Aut{\mathrm{Aut}}
\global\long\def\Out{\mathrm{\mathrm{Out}}}
\global\long\def\F{\mathbb{F}}
\global\long\def\Q{\mathbf{Q}}
\global\long\def\R{\mathbf{R}}
\global\long\def\Perms{\mathrm{Perms}}
\global\long\def\Fib{\mathrm{Fib}}
\global\long\def\tpt{\{3\text{ points}\}}
\global\long\def\T{\mathcal{T}}
\global\long\def\ord{\mathrm{ord}}
\global\long\def\Stab{\mathrm{Stab}}
\global\long\def\G{\mathcal{G}}

\section{Introduction}

Let $\Sigma:=P^{2}(\R)-\tpt$, the three times punctured real projective
plane. It is the fixed topological surface of interest in this paper.
Any hyperbolic structure $J$ of finite area on $\Sigma$ gives a
metric of curvature $-1$ and hence a way to measure the length of
curves. For fixed $J$, any isotopy class of nonperipheral simple
closed curve $[\gamma]$ on $\Sigma$ has a unique geodesic representative,
and we call the length of this geodesic with respect to $J$ simply
the length of $[\gamma]$.

It is known by work of Mirzakhani \cite{MIRZSIMPLE} that for a fixed
finite area hyperbolic structure $J$ on an \emph{orientable} surface
$S$, the number $n_{J}(L)$ of isotopy classes of simple closed curves
of length $\leq L$ has an asymptotic formula:
\begin{thm}[Mirzakhani]\label{mirztheorem}

\[
n_{J}(L)=cL^{d}+o(L^{d})
\]
where $c=c(J)>0$ and $d=d(S)>0$ is the \textbf{\emph{integer}}\emph{
}dimension of the space of compactly supported measured laminations
on $S$.
\end{thm}
In the case of the once punctured torus, a stronger form of Theorem \ref{mirztheorem} was obtained previously by McShane and Rivin  \cite{MR}.

An isotopy class of simple closed curve in $\Sigma$ is said to be
\emph{one sided }if cutting along this curve creates only one boundary
component, or in other words, a thickening of this curve is homeomorphic
to a M\"{o}bius band. The point of the current paper is to establish an
asymptotic formula for $n_{J}^{(1)}(L)$, the number of isotopy\emph{
}classes of \emph{one sided }simple closed curves of length $\leq L$
with respect to a given hyperbolic structure $J$ on $\Sigma$.
\begin{thm}
\label{thm:main}There is a \textbf{noninteger }parameter $\beta>0$
such that for any finite area hyperbolic structure $J$ on $P^{2}(\R)-\tpt$,
\[
n_{J}^{(1)}(L)=cL^{\beta}+o(L^{\beta})
\]

for some $c=c(J)>0$.
\end{thm}
The parameter $\beta$ appeared for the first time in the work of
Baragar \cite{BARAGAR2,BARAGAR1,BARAGAR3} in connection with the
affine varieties
\[
V_{n,a}:\quad x_{1}^{2}+x_{2}^{2}+\ldots+x_{n}^{2}=ax_{1}x_{2}\ldots x_{n}.
\]
These varieties have a rich automorphism group that contains an embedded
copy of $\G:=C_{2}^{*n}$, the free product of a cyclic group of size
$2$ with itself $n$ times. Baragar proved that for $o\in V(\Z)$
the following limit exists and is independent of $o$:
\[
\lim_{R\to\infty}\frac{\log|\G.o\cap B_{\ell^{\infty}}(R)|}{\log\log R}=\beta(n)>0.
\]
The variety $V_{4,1}$ was connected to the\emph{ }Teichm\"{u}ller space
of $\Sigma$ by Huang and Norbury in \cite{HN}. The value $\beta$
of Theorem \ref{thm:main} is therefore $\beta:=\beta(4)$ that Baragar
estimated to be in the range
\[
2.430<\beta(4)<2.477.
\]

Using Baragar's result, Huang and Norbury proved in \cite{HN} for
an arbitrary hyperbolic structure $J$ on $\Sigma$ that\footnote{This statement corrects the statement in \cite[Theorem 3]{HN}.}
\[
\lim_{L\to\infty}\frac{\log n_{J}^{(1)}(L)}{\log L}=\beta.
\]

A true asymptotic count for the integer points $V_{n,a}(\Z)$ was
obtained\footnote{In fact the paper \cite{GMR} treats slightly more general varieties
than $V_{n,a}$.} by Gamburd, Magee and Ronan in \cite[Theorem 3]{GMR}.

\begin{thm}[Gamburd-Magee-Ronan]
\label{thm:GMR-thm}Let $o\in V_{n,a}(\Z)$ and $\beta=\beta(n)$
as for Baragar \cite{BARAGAR2}. There is $c(o)>0$ such that
\[
|\G.o\cap B_{\ell^{\infty}}(R)|=c(\log R)^{\beta}+o\left((\log R)^{\beta}\right).
\]
\end{thm}
This is a strengthening of Baragar's result analogous to the main
Theorem \ref{thm:main}. It is worth noting that the type of
arguments used by Huang and Norbury in \cite{HN} would not be enough
to establish Theorem \ref{thm:main}, even using Theorem \ref{thm:GMR-thm}
as input. In the sequel we show how to combine and refine the arguments
of \cite{GMR} and \cite{HN} to prove Theorem \ref{thm:main}.

We also point out the recent preprint of Gendulphe \cite{GENDULPHE} who has begun a systematic investigation into the issues of growth rates of simple geodesics on general non-orientable surfaces.

\section{Orbits on Teichm\"{u}ller space}

The \emph{curve complex} of $\Sigma$ is the simplicial complex whose
vertices are isotopy classes of one sided simple closed curves, and
a collection of $k+1$ curves span a $k$-simplex if they pairwise
intersect once. We write $Z$ for this complex that was introduced
by Huang and Norbury in \cite{HN}, and its 1-skeleton was studied
earlier by Scharlemann in \cite{SCH}. It is a pure complex of dimension
$3$, that is, all maximal simplices are 3 dimensional. Throughout
the paper we use the notation $Z^{k}$ for the $k$-simplices of $Z$.

 The collection of all finite area hyperbolic structures on $\Sigma$
is called the \emph{Teichm\"{u}ller space of $\Sigma$ }and denoted by
$\T(\Sigma)$. It has a natural real analytic structure.

Let $V$ be the affine subvariety of $\C^{4}$ cut out by the equation
\begin{equation}
x_{1}^{2}+x_{2}^{2}+x_{3}^{2}+x_{4}^{2}=x_{1}x_{2}x_{3}x_{4}.\label{eq:variety-def}
\end{equation}
It was proven by Hu, Tan and Zhang in \cite[Theorem 1.1]{HPZ} that
the automorphism group of the complex variety $V$ is given by 
\[
\Lambda\rtimes(N\rtimes S_{4})
\]
where
\begin{enumerate}
\item $N$ is the group of transformations that change the sign of an even
number of variables.
\item $S_{4}$ is the symmetric group on $4$ letters that acts by permuting
the coordinates of $\C^{4}$.
\item $\Lambda$ is a nonlinear group generated by \emph{Markoff moves},
e.g.
\[
m_{1}(x_{1},x_{2},x_{3},x_{4})=(x_{2}x_{3}x_{4}-x_{1},x_{2},x_{3},x_{4})
\]
replaces $x_{1}$ by the other root of the quadratic obtained by fixing
$x_{2},x_{3},x_{4}$ in (\ref{eq:variety-def}). Similarly there are
moves $m_{2},m_{3},m_{4}$ that flip the roots in the other coordinates,
and $m_{1},m_{2},m_{3},m_{4}$ generate a subgroup 
\begin{equation}
\Lambda\cong C_{2}*C_{2}*C_{2}*C_{2}\label{eq:lambda_id}
\end{equation}
of $\Aut(V)$ where the $m_{i}$ correspond to the generators of the
$C_{2}$ factors.
\end{enumerate}
Since the abstract group $C_{2}^{*4}$ acts in different ways in the
sequel, we let
\[
\G:=C_{2}^{*4}.
\]
We obtain an action of $\G$ on $V(\R_{+})$ by the identification
(\ref{eq:lambda_id}).

Huang and Norbury in \cite{HN} prove that $V(\R_{+})$ can be identified
with $\T(\Sigma)$ by the following map. Let $\Delta=(\alpha_{1},\alpha_{2},\alpha_{3},\alpha_{4})$
be an ordering of a 3-simplex of $Z$. Let $\ell_{\alpha_{j}}(J)$
be the length of the geodesic representative of $\alpha_{j}$ in the
metric of $J$. Define a map 
\[
\Theta_{\Delta}(J):=\left(x_{\alpha_{1}}(J),x_{\alpha_{2}}(J),x_{\alpha_{3}}(J),x_{\alpha_{4}}(J)\right)
\]
where 
\begin{equation}
x_{\alpha_{i}}(J):=\sqrt{2\sinh\left(\frac{1}{2}\ell_{\alpha_{i}}(J)\right)}.\label{eq:x_equation}
\end{equation}
Building on work of Penner \cite{Penner}, Huang and Norbury show
\begin{thm}[{\cite[Proposition 8 and Section 2.4]{HN}}]
\label{thm:Teichmuller-identification}For any ordering of the curves
$\Delta=(\alpha_{1},\alpha_{2},\alpha_{3},\alpha_{4})$ in a 3-simplex
of $Z$, 
\begin{align*}
\Theta_{\Delta}:\T(\Sigma) & \to V(\R_{+})
\end{align*}

is a real analytic diffeomorphism.
\end{thm}
Let $Z_{\ord}^{3}$ denote tuples $(\alpha_{1},\alpha_{2},\alpha_{3},\alpha_{4})$
such that $\{\alpha_{1},\alpha_{2},\alpha_{3},\alpha_{4}\}$ is a
3-simplex of $Z$. It is more symmetric to consider instead of Theorem
\ref{thm:Teichmuller-identification}, the pairing
\begin{align*}
\langle\bullet,\bullet\rangle:T(\Sigma)\times Z_{\ord}^{3} & \to V(\R_{+}),\\
\langle J,\Delta\rangle & :=\Theta_{\Delta}(J).
\end{align*}
Huang and Norbury note for fixed $\Delta=(\alpha,\beta,\gamma,\delta)\in Z_{\ord}^{3}$
there is a unique way to `flip' each of $\alpha,\beta,\gamma,\delta$
to another one sided simple closed curve, say $\alpha'$ in the case
of $\alpha$ being flipped, so that e.g. $\Delta'=(\alpha',\beta,\gamma,\delta)$
is in $Z_{\ord}^{3}$, i.e. $\alpha'$ intersects each of $\beta,\gamma,\delta$
once. 

This yields an action of $\G$ on $Z_{\ord}^{3}$ where the generator
of the first $C_{2}$ factor always acts by flipping the first curve
and so on. Recall also the action of $\G$ on $V(\R_{+})$. The pairing
$\langle\bullet,\bullet\rangle$ is equivariant for the action of
$\G$ on the second factor:
\[
\langle J,g.(\alpha_{1},\alpha_{2},\alpha_{3},\alpha_{4})\rangle=g.\langle J,(\alpha_{1},\alpha_{2},\alpha_{3},\alpha_{4})\rangle,\quad g\in\G.
\]

\section{Dynamics of the markoff moves}

Our approach to counting relies on establishing the following properties
for points $x\in V(\R_{+})$ in various contexts.
\begin{description}
\item [{\label{enu:The-largest-entry-unique}A}] The largest entry of $x$
appears in exactly one coordinate. 
\item [{\label{enu:decrease-largest}B}] If $x_{j}$ is the largest coordinate
of $x$ then the largest entry of $m_{j}(x)$ is smaller than $x_{j}$,
that is, $(m_{j}(x))_{i}<x_{j}$ for all $i.$ 
\item [{\label{enu:increase-non-largest}C}] If $x_{j}$ is not the unique
largest coordinate of $x$ then it becomes the largest after the move
$m_{j}$, that is, $(m_{j}(x))_{j}>(m_{j}(x))_{i}$ for all $i\neq j.$ 
\end{description}
We will have use for the following theorem due to Hurwitz \cite{HURWITZ},
building on work of Markoff \cite{MARKOFF}.
\begin{thm}[Markoff, Hurwitz: Infinite descent]
\label{thm:Hurwitz}If $x\in V(\Z_{+})-(2,2,2,2)$ then Properties
\textbf{A}, \textbf{B} and \textbf{C} hold for $x$.
\end{thm}
\begin{proof}
Hurwitz showed the corresponding result for the point $(1,1,1,1)\in V'(\Z_{+})$
where $V'$ is defined by
\[
V':\quad x_{1}^{2}+x_{2}^{2}+x_{3}^{2}+x_{4}^{2}=4x_{1}x_{2}x_{3}x_{4}.
\]
It is easy to check that the map $V'(\Z_{+})\to V(\Z_{+})$, $x\mapsto2x$
is a bijection. 
\end{proof}
\begin{cor}
\label{cor:uniqueness}Every $x$ in $V(\Z_{+})$ has every entry
$x_{j}\geq2$ and is obtained by a unique series of nonrepeating $m_{j}$
from $(2,2,2,2)$.
\end{cor}
The following observation will be used several times in the remainder
of the section.
\begin{lem}
\label{lem:discreteness}For any $o\in V(\R_{+})$, the coordinates
of $\G.o$ form a discrete set.
\end{lem}
\begin{proof}
For fixed $\Delta\in Z_{\ord}^{3}$ let $J$ be such that $\langle J,\Delta\rangle=o$.
Then $\G.o=\langle J,\G.\Delta\rangle$ and the coordinates are all
obtained as $\sqrt{2\sinh\left(\frac{1}{2}\ell\right)}$ where $\ell$
is the length of some one sided simple closed curve in $\Sigma$ w.r.t.
$J$. Since these values of $\ell$ are discrete in $\R_{+}$ and
$\sinh^{1/2}$ has bounded below derivative in $\R_{+}$ we are done.
\end{proof}
Lemma \ref{lem:discreteness} has the following fundamental consequence
that makes our counting arguments work.
\begin{lem}
\label{lem:alpha_bound_below}For every point $o\in V(\R_{+})$ there
is some $\epsilon=\epsilon(o)>0$ such that for all $x\in\G.o$ we
have
\[
x_{i}x_{j}\geq2+\epsilon,\quad1\leq i<j\leq4.
\]
\end{lem}
\begin{proof}
Let $x\in\G.o$ and without loss of generality suppose $x_{1}\leq x_{2}\leq x_{3}\leq x_{4}$.
Since from (\ref{eq:variety-def})
\[
x_{1}x_{2}x_{3}x_{4}-x_{3}^{2}-x_{4}^{2}=x_{1}^{2}+x_{2}^{2}>0
\]
we obtain 
\[
x_{1}x_{2}x_{3}x_{4}>x_{3}^{2}+x_{4}^{2}\geq2x_{3}x_{4}
\]
implying that $x_{1}x_{2}>2$. Since $x_{1}$ and $x_{2}$ are related
by (\ref{eq:x_equation}) to lengths of simple closed curves with
respect to a hyperbolic structure $J=J(o)$, they take on discrete
values in $\G.o$ that are bounded away from $0$ and so the possible
values of $x_{1}x_{2}$ with $2<x_{1}x_{2}<3$ are discrete.
\end{proof}
We also need the following theorem that establishes Theorem \ref{thm:Hurwitz}
for an arbitrary orbit of $\G$, outside a compact set depending on
the orbit.
\begin{thm}
\label{prop:There-is-a-set-K}For given $o\in V(\R_{+})$, there is
a compact $S_{4}$-invariant set $K=K(\G.o)\subset V(\R_{+})$ such
that Properties \textbf{A}, \textbf{B} and \textbf{C} hold for $x\in\G.o-K$.
Call a move that takes place at a non-(uniquely largest) entry of
$x$\textbf{ outgoing. }The set $\G.o-K$ is preserved under outgoing
moves.
\end{thm}
\begin{proof}
Fix $o$ throughout the proof. Lemma \ref{lem:alpha_bound_below}
tells us that for some $\epsilon>0$, $x_{i}x_{j}\geq2+\epsilon$
for all $x\in\G.o$ and $1\leq i<j\leq4.$ Let $K_{0}:=\{(x_{1},x_{2},x_{3},x_{4})\in V(\R_{+})\::\|x\|_{\infty}\leq10\:\}.$
We will choose $K$ such that $K_{0}\subset K$. 

\textbf{A. }Take $x\in\G.o$. We'll prove something stronger than
property \textbf{A} for suitable choice of $K$, and use this later
in the proof. Suppose for simplicity $x_{1}\leq x_{2}\leq x_{3}\leq x_{4}.$
Write $x_{4}=x_{3}+\delta$ and assume $\delta<\delta_{0}$ where
$\delta_{0}<1$ is small enough to ensure
\begin{equation}
(1+\delta x_{3}^{-1})x_{1}x_{2}-1-(1+\delta x_{3}^{-1})^{2}>\frac{\epsilon}{2}\label{eq:e-on-two-lower}
\end{equation}
given $x_{3}\geq9$ (which we know to be the case since $x\notin K_{0}$).
We will enlarge $K_{0}$ so that this is a contradiction. From (\ref{eq:variety-def})
\begin{equation}
x_{1}^{2}+x_{2}^{2}+x_{3}^{2}(1+(1+\delta x_{3}^{-1})^{2}-(1+\delta x_{3}^{-1})x_{1}x_{2})=0,\label{eq:nearly-equal}
\end{equation}
so $x_{3}^{2}(3+(1+\delta x_{3}^{-1})^{2}-(1+\delta x_{3}^{-1})x_{1}x_{2})>0$
and hence
\begin{equation}
x_{1}x_{2}\leq\frac{3+(1+\delta x_{3}^{-1})^{2}}{(1+\delta x_{3}^{-1})}<5\label{eq:5bound}
\end{equation}
given $x\notin K_{0}$ (so $x_{3}\geq9)$ and the assumption $\delta<1$.
On the other hand (\ref{eq:e-on-two-lower}) and (\ref{eq:nearly-equal})
now imply that if $\eta>0$ is a lower bound for all coordinates of
$\G.o$ then
\[
x_{3}^{2}\leq\frac{2}{\epsilon}\left(x_{1}^{2}+x_{2}^{2}\right)\leq\frac{4}{\epsilon}x_{2}^{2}\leq\frac{4}{\eta^{2}\epsilon}x_{1}^{2}x_{2}^{2}\leq\frac{100}{\eta^{2}\epsilon},
\]
where the last inequality is from (\ref{eq:5bound}).

Now let 
\begin{align*}
K_{1} & =\{x\in V(\R_{+})\::\:\|x\|_{\infty}\leq\frac{10}{\eta\sqrt{\epsilon}}+1\}\cup K_{0}.
\end{align*}
\emph{We proved there is $\delta=\delta(o)$ such that for $x\in\G.o-K_{1}$,
there is an entry of $x$ that is $\geq\delta$ more than all the
other entries.}

\textbf{B. }Take $x\in\G.o-K_{1}$ with $x_{1}\leq x_{2}\leq x_{3}<x_{4}.$
We follow the method of Cassels \cite[pg. 27]{CASSELS}. Consider
the quadratic polynomial
\[
f(T)=T^{2}-x_{1}x_{2}x_{3}T+x_{1}^{2}+x_{2}^{2}+x_{3}^{2}.
\]
Then $f$ has roots at $x_{4}$ and $x'_{4}$ where $x'_{4}$ is the
last entry of $m_{4}(x)$. Property \textbf{B }holds at $x$ unless
$x_{3}<x_{4}\leq x_{4}'$, in which case $f(x_{3})>0$ giving
\[
x_{3}^{2}(4-x_{1}x_{2})\geq x_{3}^{2}(2-x_{1}x_{2})+x_{1}^{2}+x_{2}^{2}>0.
\]
Therefore $x_{1}x_{2}<4$. By discreteness of the coordinates of $\G.o$
this means there are finitely many possibilities for $x_{1}$ and
$x_{2}$. Now $x'_{4}\geq x_{4}$ directly implies
\[
x_{1}x_{2}x_{3}\geq2x_{4}
\]
so $2x_{4}^{2}\leq x_{1}x_{2}x_{3}x_{4}=x_{1}^{2}+x_{2}^{2}+x_{3}^{2}+x_{4}^{2}$
and so
\[
(x_{4}+x_{3})(x_{4}-x_{3})\leq x_{1}^{2}+x_{2}^{2}\leq M
\]
for some $M$ depending on the finitely many possible values for $x_{1},x_{2}$.
Since we know $x_{4}-x_{3}\geq\delta$ we obtain
\[
x_{4}+x_{3}\leq\frac{M}{\delta}
\]
so $x_{3}\le x_{4}\leq M\delta^{-1}$. Let $K_{2}:=\{x\in V(\R_{+})\::\:\|x\|_{\infty}\leq M\delta^{-1}\}\cup K_{1}$.
This establishes \textbf{B }for $x\in\G.o-K_{2}$.

\textbf{C. }Take $x\in\G.o$ with $x_{1}\leq x_{2}\leq x_{3}\leq x_{4}.$
Then for $1\leq j\leq3$, 
\[
(m_{j}(x))_{j}=\frac{x_{1}x_{2}x_{3}x_{4}}{x_{j}}-x_{j}\geq x_{1}x_{2}x_{4}-x_{3}=x_{4}\left(x_{1}x_{2}-\frac{x_{3}}{x_{4}}\right)\geq x_{4}(1+\epsilon)>x_{4}.
\]
by Lemma \ref{lem:alpha_bound_below}. This establishes \textbf{C}.

We established \textbf{A, B }for $x\in\G.o-K$ with $K=K_{2}$, and
\textbf{C }for any $x\in\G.o$. It is clear from the previous that
$\G.o-K$ is stable under outgoing moves.
\end{proof}

\section{The topology of the curve complex }

\label{sec:curve-complex}

Our first goal in this section is to prove the following topological
theorem. Let $G$ be the graph whose vertices are $3$-simplices $\{\alpha,\beta,\gamma,\delta\}$
of $Z$ with an edge between two vertices if they share a dimension
$2$ face. 
\begin{thm}
\label{thm:reg-tree}$G$ is a 4-regular tree.
\end{thm}
This theorem is stated without proof in \cite[pg. 9]{HN}, and then
used throughout the rest of the paper \cite{HN}. We have been careful
here only to use results from \cite{HN} that are deduced independently
from Theorem \ref{thm:reg-tree}, to avoid circularity.

We prove Theorem \ref{thm:reg-tree} in two steps, using the following
theorem of Scharlemann:
\begin{thm}[{\cite[Theorem 3.1]{SCH}}]
\label{thm:The--skeleton-of}The $1$-skeleton of $Z$ is the $1$-skeleton
of the complex obtained by repeated stellar subdivision of the dimension
$2$ faces of a tetrahedron. 
\end{thm}
\begin{cor}
$Z$ is connected.
\end{cor}
Recall that the clique complex of a graph $H$ has the same vertex
set as $H$ and a $k$-simplex for each clique (complete subgraph)
of $H$ of size $k+1$. Note that $Z$ is the clique complex of its
1-skeleton.
\begin{figure}
\includegraphics{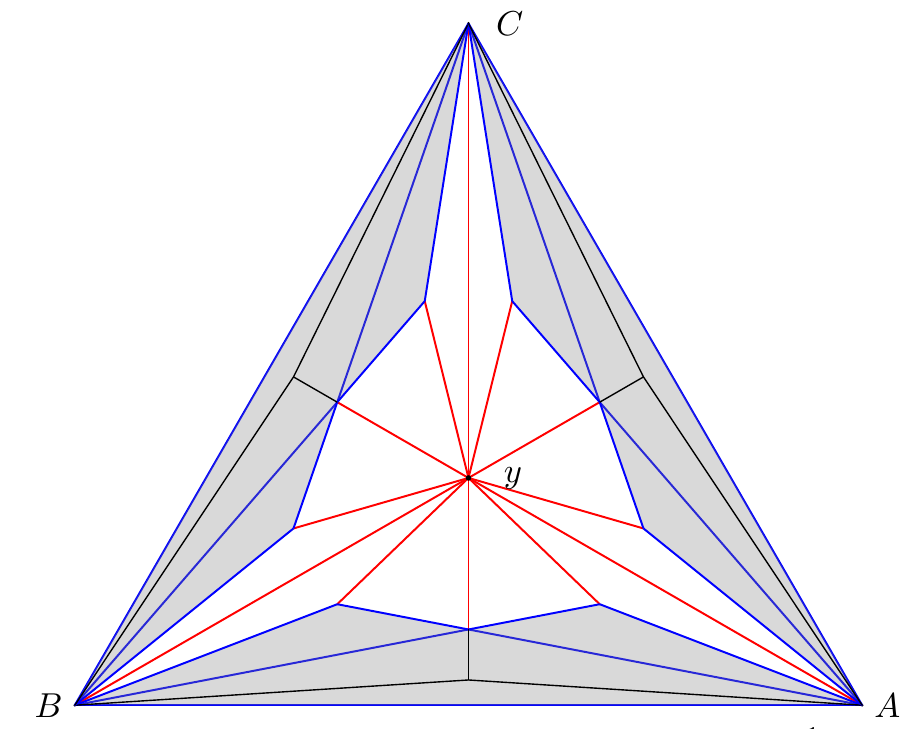}

\caption{\label{fig:vertex_link}This shows 2 generations of repeated stellar
subdivision of $T_{1}(y),T_{2}(y)$ and $T_{3}(y)$. The link in $Z$ of the vertex $y$ can be identified with
the closure of the shaded region, together with a triangle with vertices
$A,B,C$.}
\end{figure}

\begin{figure}
\includegraphics{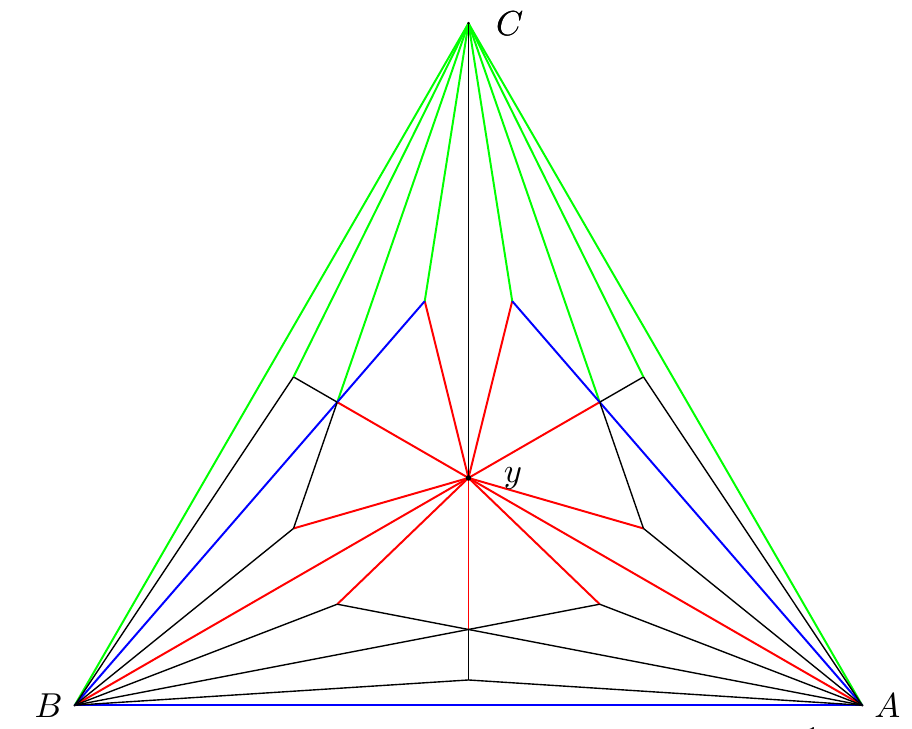}

\caption{\label{fig:edge_link}This shows 2 generations of repeated stellar
subdivision of $T_{1}(y),T_{2}(y)$ and $T_{3}(y)$, the link in $Z$ of the
edge between $y$ and $C$ can be identified
with the closure of the blue edges.}
\end{figure}

\begin{lem}
\label{lem:LINKS}The link of a vertex or edge in $Z$ is contractible.
In particular all links of simplices of codimension $>1$ in $Z$
are connected.
\end{lem}
\begin{proof}
Let $Y$ be the 1 dimensional subcomplex of Theorem \ref{thm:The--skeleton-of}.
Let $\Delta^{(2)}$ be the $2$-skeleton of a standard $3$-simplex
$\Delta.$

Since $Z$ is the clique complex of $Y$ it is possible to characterize
links of simplices in $Z$ purely in terms of cliques in $Y.$ Precisely,
the link of a simplex $s$ in $Z$ is the collection of all cliques
in $Y$ that are disjoint from $s$ but that together with $s$ form
a clique.

We view $Y$ as a graph drawn on $|\Delta^{(2)}|$. For every vertex
$y$ of $Y$ there are $3$ other vertices $A(y),B(y),C(y)$ of $Y$
such that
\begin{enumerate}
\item $y,A,B$ and $C$ are a clique in $Y$.
\item Every vertex adjacent to $y$ in $Y$ is contained in one of the triangles
$T_{1}(y)$,$T_{2}(y)$ or $T_{3}(y)$ in $|\Delta^{(2)}|$ with vertices
$(A,B,y)$, $(B,C,y)$ or $(A,C,y)$ respectively. In the case $y$
is a vertex of $\Delta$, these triangles are faces of $|\Delta|.$
Otherwise they are all contained in the same face of $|\Delta|$ that
contains $y$.
\item More precisely, every vertex of $Y$ adjacent to $y$, and all edges
between these vertices, are generated by repeated stellar subdivision
of the triangles $T_{1}$,$T_{2}$ and $T_{3}$ together with the
edges and vertices of the $T_{i}$.
\end{enumerate}
These observations mean that all links of vertices of $Z$ look the
same and can be calculated by drawing the same picture. Similarly
all links of edges can be calculated in the same way.

Figure \ref{fig:vertex_link} shows the link of the central vertex
$y$, truncating after 2 iterations of the stellar subdivision. The
red edges are incident with $y$. The blue edges are edges not incident
with $y$ but whose vertices are adjacent to $y$. Cliques in the
link of $y$ in $Z$ are cliques relative to blue edges. We observe
that since the drawing of this part of $Y$ is planar, the blue cliques
of size 3 other than $\{A,B,C\}$ bound nonoverlapping regions, so
we can identify the geometric realization of the link of $y$ in $Z$
with the closure of the shaded triangles here, together with an extra
triangle with vertices $A,B,C$. This geometric realization is visibly
a topological disc. The effect of iterating stellar subdivision is
that the shaded region encroaches inwards, but its homotopy type doesn't
change.

Similarly the link of the edge between $y$ and $C$ is approximated
in Figure \ref{fig:edge_link}. Green edges emanate from $C$ and
red emanate from $y$. A blue edge has both vertices adjacent to both
$y$ and $C$ (i.e. having incident red and green edges). The closure
of the blue edges hence approximates the link of $\{y,C\}$ in $Z$
and is homeomorphic to a line segment. Iterating stellar subdivision
extends the segment on both signs and as before, the homotopy type
doesn't change. 
\end{proof}
Since $Z$ is connected we obtain the following consequence of Lemma
\ref{lem:LINKS} (cf. Hatcher \cite[pg. 3, proof of Corollary]{H}).
The basic idea is to use Lemma \ref{lem:LINKS} to inductively deform
any path in $Z$ away from codimension $>1$ simplices.
\begin{cor}
\label{cor:-is-connected.}$G$ is connected.
\end{cor}
\begin{lem}
\label{lem:G-is-acyclic.}G is acyclic.
\end{lem}
\begin{proof}
Suppose $G$ has a cycle, so that there is a series of nonrepeating
flips that map a vertex $\Delta_{0}=\{\alpha,\beta,\gamma,\delta\}$
to itself. Pick the ordering $\Delta=(\alpha,\beta,\gamma,\delta)$
of this vertex. By Theorem \ref{thm:Teichmuller-identification} there
is a hyperbolic structure $J$ on $\Sigma$ so that $\langle J,\Delta\rangle=(2,2,2,2).$
The flips of $\Delta_{0}$ yield a unique nonrepeating series of flips
of $\Delta$ that in turn yield a unique \emph{nonrepeating} series
of Markoff-Hurwitz moves $m_{i}$ preserving $(2,2,2,2)$. By Corollary
\ref{cor:uniqueness} the series of flips has to be empty.
\end{proof}
These results (Corollary \ref{cor:-is-connected.} and Lemma \ref{lem:G-is-acyclic.})
conclude the proof of Theorem \ref{thm:reg-tree} since we established
$G$ is an acyclic connected graph that we also know to be 4-valent.

In the rest of this section we prove that smaller pieces of $G$ are
connected and acyclic. Specifically, for any simplex $\Delta\in Z$
we may form $G_{\Delta}$, the subgraph of $G$ induced by vertices
containing $\Delta$. For example, if $\Delta$ is a 2-simplex then
$G_{\Delta}$ has two vertices and an edge representing a flip between
them. If $\Delta$ is a 3-simplex then $G_{\Delta}$ has only one
vertex, $\Delta$. More generally,
\begin{prop}
\label{prop:local-graph-tree}For all $\Delta\subset Z$, $G_{\Delta}$
is a tree.
\end{prop}
\begin{proof}
Since $G$ is acyclic it suffices to prove $G_{\Delta}$ is connected.
We give the proof that $G_{\delta}$ is connected in the case $\delta$
is a vertex of $Z$, the case $\Delta$ is an edge is similar and
we have already discussed the other cases.

Suppose $\delta$ is a vertex of $\Delta,\Delta'\in Z^{3}$. We aim
to connect $\Delta$ to $\Delta'$ by flips that don't touch $\delta$.
Order $\Delta$ and $\Delta'$ so that $\delta$ is the final element
of each. Let $J$ be the hyperbolic structure provided by Theorem
\ref{thm:Teichmuller-identification} such that $\langle J,\Delta\rangle=(2,2,2,2)$.
Since $\Delta'=(\beta_{1},\beta_{2},\beta_{3},\delta)$, $\langle J,\Delta'\rangle=(x_{1},x_{2},x_{3},2)$
for some $x_{1},x_{2},x_{3}\in\Z.$ The infinite descent (Theorem
\ref{thm:Hurwitz}) for $V(\Z_{+})$ now yields a series of flips
that never modifies $\delta$, starts at $\Delta'$ and ends at some
$\Delta''=(\gamma_{1},\gamma_{2},\gamma_{3},\delta)\in Z_{\ord}^{3}$
with $\langle J,\Delta''\rangle=(2,2,2,2).$ Also note that by combining
Theorem \ref{thm:Hurwitz} and Theorem \ref{thm:reg-tree}, there
is a unique $\Delta_{0}\in Z^{3}$ such that $\langle J,\Delta_{0}\rangle=(2,2,2,2)$
for any ordering of $\Delta_{0}$. Therefore up to reordering, $\Delta''=\Delta_{0}=\Delta$
as required.
\end{proof}
There is a nice corollary of Proposition \ref{prop:local-graph-tree}
that may be of independent interest.
\begin{cor}
The curve complex $Z$ has the homotopy type of a point.
\end{cor}
\begin{proof}
The collection $\{G_{\Delta}:\Delta\:\:\text{\ensuremath{0}-dimensional}\}$
is a cover of $G$ by subcomplexes. The nerve of this cover can be
identified with $Z$, and each finite nonempty intersection of the
covering complexes is $G_{\Delta}$ for $\Delta$ a simplex of $Z$,
and hence is contractible by Proposition \ref{prop:local-graph-tree}.
Therefore the Nerve Theorem \cite[ch. VII, Thm. 4.4]{BROWN} applies to give the result.
\end{proof}

\section{Proof of theorem \ref{thm:main}}

Let $\Gamma$ denote the mapping class group of $\Sigma$. Mapping
classes in $\Gamma$ may permute the punctures of $\Sigma$. The group
$\Gamma$ acts simplicially on $Z$ in the obvious way.

Recall that for each 3-dimensional simplex $\Delta=\{\alpha,\beta,\gamma,\delta\}\in Z$,
there is a unique flip of $\alpha$ that produces a new simplex $\{\alpha',\beta,\gamma,\delta\}$.
Further to this, Huang and Norbury \cite{HN} construct a corresponding
unique mapping class $\gamma_{\Delta}^{1}\in\Gamma$ that maps $\{\alpha,\beta,\gamma,\delta\}$
to $\{\alpha',\beta,\gamma,\delta\}$, similarly $\gamma_{\Delta}^{2}$
performs a flip at $\beta$ and so on. 

The mapping class elements $\gamma_{\Delta}^{i}$ can be extended
to a cocycle for the group action of $\G$ on $Z_{\ord}^{3}$. In
other words, for every $\Delta\in Z_{\ord}^{3}$ and $g\in\G$ there
is a mapping class group element $\gamma(g,\Delta)$ such that $\gamma(g,\Delta)\Delta=g\Delta.$
For example, if $g_{1}$ is the generator of the first factor of $\G$
then $\gamma(g,\Delta)=\gamma_{\Delta}^{1}$.
\begin{prop}
\label{prop:gen-orbit}For any given $\Delta$, if $\tilde{\Delta}\in Z_{\ord}^{3}$
is an ordering of $\Delta$ then the map
\begin{align}
\G & \to Z^{3}\nonumber \\
g & \mapsto\gamma(g,\tilde{\Delta}).\Delta\label{eq:cocycle}
\end{align}
is a bijection.
\end{prop}
\begin{proof}

The map $g\mapsto\gamma(g,\tilde{\Delta})\Delta$ yields a graph homomorphism
from the Cayley graph of $\G$, a 4-regular tree, to $G$. Recall
that $G$ is also a 4-regular tree by Theorem \ref{thm:reg-tree}.
The homomorphism is locally injective. Therefore $g\mapsto\gamma(g,\tilde{\Delta})\Delta$
is a bijection.
\end{proof}
The next proposition allows us to pass from counting over $G$ to
counting over simple closed curves (our goal), up to finite subsets
at either side of the passage.
\begin{prop}
\label{prop:counting-over-curves} Let $J\in\T(\Sigma)$ and for arbitrary
fixed $\Delta_{0}\in Z_{\ord}^{3}$ let $o:=\langle J,\Delta_{0}\rangle$.
Let $K$ be a compact $S_{4}$-invariant subset of $V(\R_{+})$ containing
the set $K(\G.o)$ from Theorem \ref{prop:There-is-a-set-K}. Since
$K$ is $S_{4}$-invariant, the condition $\langle J,\Delta\rangle\notin K$
is independent of the ordering of $\Delta\in Z^{3}$, and so well
defined. The map 
\begin{align*}
\Phi:\{\:\Delta & \in Z^{3}\::\:\langle J,\Delta\rangle\notin K\:\}\to Z^{0},
\end{align*}
\begin{align}
\Phi:\{\alpha_{1},\alpha_{2},\alpha_{3},\alpha_{4}\} & \mapsto\{\alpha_{i}\}\::\:\ell_{\alpha_{i}}(J)=\max_{1\leq j\leq4}\ell_{\alpha_{j}}(J)\label{eq:longest-curve}
\end{align}

is a well defined injection whose image is all but finitely many elements
of $Z^{0}.$
\end{prop}
\begin{proof}
That $\Phi$ is well defined is immediate from Theorem \ref{prop:There-is-a-set-K},
Property \textbf{A}. 

Suppose $\delta\in Z^{0}$ is the longest curve in each of $\Delta,\Delta'$
with respect to $J$, with $\langle J,\Delta\rangle,\langle J,\Delta'\rangle\notin K$.
By Proposition \ref{prop:local-graph-tree} there is a series of flips
taking $\Delta$ to $\Delta'$ and never modifying $\delta$. By Property
\textbf{C} of Theorem \ref{prop:There-is-a-set-K}, the first flip
creates a curve longer than $\delta$ w.r.t. $J$. This continues,
since $\G.o-K$ is stable under outgoing moves, and it is therefore
impossible to reach $\Delta'\neq\Delta$ since $\delta$ is the largest
curve of $\Delta'$, but not of any intermediate simplex of the sequence
that was generated. This establishes injectivity of $\Phi$.

As for the final statement that the image of (\ref{eq:longest-curve})
misses only finitely many curves, let $\delta\in Z^{0}$. We aim to
find $(\alpha',\beta',\gamma,\delta)$ for which $\delta$ is the
longest curve with respect to $J$. Say that $\delta$ is \emph{bad
}if\textbf{ $\langle J,\Delta\rangle\in K$ }for some $\Delta$ containing
$\delta$. Otherwise say $\delta$ is \emph{good}. Since $K$ is compact,
and the set of lengths of one sided simple closed curves in $J$ is
discrete, there are only finitely many bad $\delta$. We will prove
all good $\delta$ are in the image of $\Phi$. For good $\delta$,
begin with any $\Delta\in Z_{\ord}^{3}$ such that $\langle J,\Delta\rangle\notin K$
and $\delta$ is last in $\Delta$. If $\delta$ is the longest curve
of $\Delta$ with respect to $J$ then we are done. Otherwise let
$(x_{1},x_{2},x_{3},x_{4})=\langle J,\Delta\rangle$. Using Property
\textbf{B }of Theorem \ref{prop:There-is-a-set-K}, apply moves at
the largest entries of $(x_{1},x_{2},x_{3},x_{4})$ (which do not
correspond to $\delta$) until $\delta$ becomes the longest curve.
The resulting $(y_{1},y_{2},y_{3},y_{4})=\langle J,\Delta'\rangle$
cannot be in $K$, so we are done since $\delta=\Phi(\Delta').$
\end{proof}

\vspace{-1cm}
We have put all the pieces in place to use the methods of Gamburd,
Magee and Ronan \cite{GMR} to prove Theorem \ref{thm:main}. We now
give an overview of the method of \cite{GMR} and explain how what
we have already proved extends the method to the current setting.

\textbf{Step 1. }\emph{(loc. cit.)} begins with a compact set $K$
such that for $x\in\G.o-K$, properties \textbf{A}, \textbf{B}, and
\textbf{C }hold. Here, we take $K$ to be the set provided by Theorem
\ref{prop:There-is-a-set-K}. It is then deduced from \textbf{A},\textbf{
B},\textbf{ }and \textbf{C }that the number of distinct entries of
$x\in\G.o-K$ cannot decrease during an outgoing move. There is a
further regularization of $K$ in \cite[Section 2.4]{GMR}, by adding
to $K$ a large ball $B_{\ell^{\infty}}(R)$ if necessary, in order
to assume that if for example $x_{1}\leq x_{2}\leq x_{3}\leq x_{4}$
with $(x_{1},x_{2},x_{3},x_{4})\in\G.o-K$ then
\[
x_{3}\geq\frac{1}{2}x_{4}^{\frac{1}{3}},
\]
\[
\frac{3\log(1-2x_{4}^{-\frac{1}{3}})-3\log2}{\log x_{4}}\geq-\frac{1}{2},
\]
and $x_{4}\geq10.$ These inequalities play a role in technical estimates
throughout the proof, in particular, the proof of \cite[Lemma 21]{GMR}.
It is possible to increase $K$ to ensure these hold (and the corresponding
inequalities for other ordering of the coordinates of $x$) for the
same reasons as in \emph{(loc. cit.). }Also, without loss of generality,
$o\in K$.

\textbf{Step 2. }Recall the quantity $n_{J}^{(1)}(L)$ from our main
Theorem \ref{thm:main}. Fix $\Delta_{0}\in Z_{\ord}^{3}$ and let
$o:=\langle J,\Delta_{0}\rangle$. Let $K$ be the enlarged compact
set from Step 1.

Putting Propositions \ref{prop:gen-orbit} and \ref{prop:counting-over-curves}
(for the the current $K$) together gives us
\begin{align}
n_{J}^{(1)}(L) & :=\sum_{\alpha\in Z^{0}}\mathbf{1}\{\ell_{\alpha}(J)\leq L\}\nonumber \\
\text{(Proposition \ref{prop:counting-over-curves})} & =\sum_{\Delta:\:\langle J,\Delta\rangle\notin K}\mathbf{1}\left\{ \max\langle J,\tilde{\Delta}\rangle\leq\sqrt{2\sinh\left(\frac{1}{2}L\right)}\right\} +O_{J}(1)\nonumber \\
\text{(Proposition \ref{prop:gen-orbit})} & =\sum_{g\in\G:\:\G.o\not\notin K}\mathbf{1}\left\{ \max g.o\leq\sqrt{2\sinh\left(\frac{1}{2}L\right)}\right\} +O_{J}(1),\label{eq:count}
\end{align}

where for $\Delta\in Z^{3}$ we wrote $\tilde{\Delta}$ for an arbitrary
lift of $\Delta$ to $Z_{\ord}^{3}$. Since 
\[
\sqrt{2\sinh\left(\frac{1}{2}L\right)}=\sqrt{e^{L/2}-e^{-L/2}}=e^{L/4}(1+O(e^{-L}))
\]
the required asymptotic formula for (\ref{eq:count}) as $L\to\infty$
will follow from an estimate of the form
\begin{equation}
\sum_{g\in\G:\:\G.o\not\notin K}\mathbf{1}\{\max g.o\leq e^{L}\}=c(o)L^{\beta}+o(L^{\beta}).\label{eq:massaged_count}
\end{equation}

Note as in \cite{GMR} that the set $\G.o-K$ breaks up into a finite
union 
\[
\G.o-K=\cup_{i=1}^{N}\O_{i}
\]
where each $\O_{i}$ is the orbit of a point $o_{i}\in\G.o-K$ under
outgoing moves. The points $o_{i}$ are each one move outside of $K$.
The fact there are finitely many $o_{i}$ requires the discreteness
of $\G.o$ and the compactness of $K$. Each $\O_{i}$ has the form
\[
\O_{i}=\{m_{j_{M}}\ldots m_{j_{3}}m_{j_{2}}m_{j_{1}}o_{i}\::\:M\geq0,\:j_{i}\neq j_{i+1},j_{1}\neq j_{0}(i)\}
\]
where $j_{0}(i)$ is such that $m_{j_{0}(i)}o_{i}\in K$, or in other
words, $m_{j_{0}(i)}$ is not outgoing on $o_{i}$. It can be deduced
from \textbf{A, B, C} and preceding remarks\textbf{ }that each orbit
$\O_{i}$ can be identified with a subset $\G_{i}\subset\G$ via a
bijection
\[
g\in\G_{i}\mapsto g.o\in\O_{i}.
\]
Moreover the $\G_{i}$ are disjoint. Therefore
\begin{equation}
\begin{aligned}\sum_{g\in\G:\:\G.o\not\notin K}\mathbf{1}\{\max g.o\leq e^{L}\} & = & \sum_{i=1}^{N}\sum_{g\in\G_{i}}\mathbf{1}\{\max g.o\leq e^{L}\}\\
 & = & \sum_{i=1}^{N}\sum_{x\in\O_{i}}\mathbf{1}\{\max x\leq e^{L}\}.
\end{aligned}
\label{eq:temp1}
\end{equation}

This reduces the count for $n_{J}^{(1)}(L)$ to a count for each of
a finite number of orbits under outgoing moves in a region where \textbf{A},
\textbf{B} and \textbf{C} hold.

\textbf{Step 3. }The methods of \cite{GMR} now take over, with one
important thing to point out. A version of Lemma \ref{lem:alpha_bound_below}
is crucially used during the proof of \cite[Lemma 20]{GMR}. In that
instance \cite{GMR} can make a better bound than we have\footnote{Since in \cite{GMR} we were concerned with integer points, this meant
after ruling out certain special cases, it allowed us to take $\epsilon=1$
in Lemma \ref{lem:alpha_bound_below}.}, but what is really important is the existence of the uniform $\epsilon>0$
in Lemma \ref{lem:alpha_bound_below}. This establishes a weaker,
but qualitatively the same, version of \cite[Lemma 20]{GMR} that
plays the same role in the proof. The rest of the arguments of \cite{GMR}
go through without change to establish
\begin{thm}[Gamburd-Magee-Ronan, adapted]
\label{thm:GMR-addapted}For each $1\leq i\leq N$ there is a constant
$c(\O_{i})>0$ such that
\[
\sum_{x\in\O_{i}}\mathbf{1}\{\max x\leq e^{L}\}=c(\O_{i})L^{\beta}+o(L^{\beta}).
\]
\end{thm}
Using Theorem \ref{thm:GMR-addapted} in (\ref{eq:temp1}) completes
the proof of Theorem \ref{thm:main}.

\bibliographystyle{alpha}

Michael Magee,\\
Department of Mathematical Sciences,\\
Durham University,\\
Lower Mountjoy, DH1 3LE Durham,\\
United Kingdom \\
\tt{michael.r.magee@durham.ac.uk}
\end{document}